\documentclass[a4paper]{amsart}
\usepackage{graphicx}
\usepackage{amsmath}
\usepackage{amssymb}
\usepackage{oldgerm}
\usepackage[all]{xy}

\newcommand{\Hom}{\operatorname{Hom}\nolimits}
\renewcommand{\Im}{\operatorname{Im}\nolimits}

\newcommand{\Ker}{\operatorname{Ker}\nolimits}

\newcommand{\Tor}{\operatorname{Tor}\nolimits}
\newcommand{\Ext}{\operatorname{Ext}\nolimits}

\newcommand{\HH}{\operatorname{HH}\nolimits}

\newcommand{\op}{\operatorname{op}\nolimits}

\newcommand{\e}{\operatorname{e}\nolimits}

\newcommand{\cha}{\operatorname{char}\nolimits}

\newtheorem{theorem}{Theorem}[section]

\newtheorem{proposition}[theorem]{Proposition}
\theoremstyle{definition}

\theoremstyle{definition}

\theoremstyle{definition}

\theoremstyle{definition}

\theoremstyle{definition}

\theoremstyle{definition}

\theoremstyle{remark}

\theoremstyle{definition}

\theoremstyle{definition}

\begin{document}
\title{(Co)homology of quantum complete intersections}
\author{Petter Andreas Bergh \& Karin Erdmann}
\address{Petter Andreas Bergh \newline Mathematical Institute \\ 24-29
  St.\ Giles \\ Oxford \\ OX1 3LB \\ United Kingdom \newline
  \emph{Present address:} \newline Institutt for matematiske fag \\
  NTNU \\ N-7491 Trondheim \\ Norway}
\email{bergh@math.ntnu.no}
\address{Karin Erdmann \newline Mathematical Institute \\ 24-29
  St.\ Giles \\ Oxford \\ OX1 3LB \\ United Kingdom}
\email{erdmann@maths.ox.ac.uk}

\date{\today}

\dedicatory{Dedicated to Lucho Avramov on the occasion of his sixtieth birthday}

\subjclass[2000]{16E40, 16S80, 16U80, 81R50}

\keywords{Hochschild (co)homology, quantum complete intersections}

\maketitle

\begin{abstract}
We construct a minimal projective bimodule resolution for every
finite dimensional quantum complete intersection of codimension two.
Then we use this resolution to compute both the Hochschild
cohomology and homology for such an algebra. In particular, we show
that the cohomology vanishes in high degrees, while the homology is
always nonzero.
\end{abstract}

\section{Introduction}\label{intro}

The notion of quantum complete intersections originates from work by
Manin (cf.\ \cite{Manin}), who introduced the concept of \emph{quantum
symmetric algebras}. These algebras were used by Avramov, Gasharov
and Peeva in \cite{Avramov} to study modules behaving homologically
as modules
over commutative complete intersections. In particular, they
introduced \emph{quantum regular sequences} of endomorphisms of
modules, thus generalizing the classical notion of regular sequences.

In \cite{Benson}, Benson, Erdmann and Holloway defined and studied a
new rank variety theory for modules over finite dimensional quantum
complete intersections. For this theory to work, it is essential
that the commutators defining the quantum complete intersection be
roots of unity, so that a linear combination of the generators
behave itself as a generator. In this setting, at least for quantum
complete intersections of codimension two, the Hochschild cohomology
ring is infinite dimensional, and a priori there might be
connections between rank varieties and the support varieties defined
by Snashall and Solberg (cf.\ \cite{Snashall}, \cite{Erdmann}).

Whether or not the higher Hochschild cohomology groups of a finite
dimensional algebra of infinite global dimension can vanish, known
as ``Happel's question", was unknown until the appearance of
\cite{Buchweitz}. In that paper, the authors constructed a four
dimensional selfinjective algebra whose total Hochschild cohomology
is five dimensional, thus giving a negative answer to Happel's
question. The algebra they constructed is the smallest possible
noncommutative quantum complete intersection.

In this paper we study finite dimensional quantum complete
intersections of codimension two. For such an algebra, we construct
a minimal projective bimodule resolution, and use this to compute
the Hochschild homology and cohomology. In particular, we show that
the higher Hochschild cohomology groups vanish if and only if the
commutator element is not a root of unity, whereas the Hochschild
homology groups never vanish. Thus we obtain a large class of
algebras having the same homological properties as the algebra used
in \cite{Buchweitz}.

\section{The minimal projective resolution}\label{resolution}

Throughout this paper, let $k$ be a field and $q \in k$ a nonzero
element. In the main results this element is assumed not to be a root
of unity, implying indirectly that $k$ is an infinite field. We fix
two integers $a,b \ge 2$, and denote by $A$ the $k$-algebra
$$A = k \langle X,Y \rangle / (X^a,XY-qYX,Y^b).$$
This is a finite dimensional algebra of dimension $ab$, and it is
justifiably a quantum complete intersection of codimension $2$; it is
the quotient
of the quantum symmetric algebra $k \langle X,Y \rangle / (XY-qYX)$ by
the quantum regular sequence $X^a,Y^b$. We denote the generators of
$A$ by $x$ and $y$, and use the set
$$\{ y^i x^j \}_{0 \leq i < b, \hspace{1mm} 0 \leq j < a}$$
as a $k$-basis. The opposite algebra of $A$ is denoted by $A^{\op}$,
and the enveloping algebra $A \otimes_k A^{\op}$ by $A^{\e}$.

We now construct explicitly a minimal projective bimodule resolution
$$\mathbb{P} \colon \dots \to P_2 \xrightarrow{d_2} P_1 \xrightarrow{d_1} P_0
\xrightarrow{\mu} A \to 0,$$ in which $P_n$ is  free of rank $n+1$,
viewing the bimodules as left $A^{\e}$-modules. The generators $1
\otimes 1$ of $P_n$ are labeled $\epsilon (i,j)$ for $i,j \ge 0$,
such that
$$P_n = \bigoplus_{i+j=n} A^{\e} \epsilon (i,j).$$
For each $s \geq 0$ define the following four elements of $A^{\e}$:
\begin{eqnarray*}
\tau_1 (s) & = & q^s (1 \otimes x) - (x \otimes 1) \\
\tau_2 (s) & = & (1 \otimes y) - q^s (y \otimes 1) \\
\gamma_1 (s) & = & \sum_{j=0}^{a-1} q^{js} (x^{a-1-j} \otimes
x^j) \\
\gamma_2 (s) & = & \sum_{j=0}^{b-1} q^{js} (y^j \otimes y^{b-1-j}).
\end{eqnarray*}
Let $P_0 \xrightarrow{\mu} A$ be the multiplication map $w \otimes
z \mapsto wz$. The kernel of this map is generated by $\tau_1(0)$
and $\tau_2(0)$. Now let $R_1$ and $R_2$ be the commutative
subalgebras of $A$ generated by $x$ and $y$, respectively. The
annihilator of $\tau_i(0)$, viewed as an element of $R_i^{\e}$, is
$\gamma_i(0)$, and the complex
$$\cdots \to R_i^{\e} \xrightarrow{\tau_i(0)} R_i^{\e}
\xrightarrow{\gamma_i(0)} R_i^{\e} \xrightarrow{\tau_i(0)}
R_i^{\e} \xrightarrow{\mu} R_i \to 0$$ is a minimal projective
bimodule resolution of $R_i$ (cf.\ \cite{Holm}). Now for $i=1,2$,
define an algebra automorphism $R_i^{\e} \xrightarrow{\sigma_i}
R_i^{\e}$ by
$$\sigma_1 \colon \begin{array}{l}
               x \otimes 1 \mapsto x \otimes 1 \\
               1 \otimes x \mapsto q(1 \otimes x)
             \end{array}
\hspace{.5cm} \sigma_2 \colon \begin{array}{l}
                               y \otimes 1 \mapsto q(y \otimes1)\\
                               1 \otimes y \mapsto 1 \otimes y.
                               \end{array}$$
When we twist the above resolution by the automorphism
$\sigma_i^s$ for some $s \geq 0$, then multiplication by
$\tau_i(0)$ and $\gamma_i(0)$ become multiplication by $\tau_i(s)$
and $\gamma_i(s)$, respectively. We denote this twisted resolution
by ${\bf{R}}_i(s)$.

We now define a double complex
$$\xymatrix{
\vdots \ar[d] & \vdots \ar[d] & \vdots \ar[d] \\
A^{\e} \epsilon (0,2) \ar[d] & A^{\e} \epsilon (1,2) \ar[d] \ar[l]
 & A^{\e} \epsilon (2,2) \ar[d] \ar[l] & \cdots \ar[l] \\
A^{\e} \epsilon (0,1) \ar[d] & A^{\e} \epsilon (1,1) \ar[d] \ar[l]
 & A^{\e} \epsilon (2,1) \ar[d] \ar[l] & \cdots \ar[l] \\
A^{\e} \epsilon (0,0) & A^{\e} \epsilon (1,0) \ar[l]
 & A^{\e} \epsilon (2,0) \ar[l] & \cdots \ar[l]}$$
whose total complex $\mathbb{P}$ turns out to be the projective
bimodule resolution we are seeking. Along row $2s$ we use the
resolution ${\bf{R}}_1(bs)$, and along row $2s+1$ we use the
resolution ${\bf{R}}_1(bs+1)$. Explicitly, the row maps are given by
\begin{eqnarray*}
\epsilon (2r,2s) & \mapsto & \gamma_1(bs) \epsilon (2r-1,2s) \\
\epsilon (2r+1,2s) & \mapsto & \tau_1(bs) \epsilon (2r,2s) \\
\epsilon (2r,2s+1) & \mapsto & \gamma_1(bs+1) \epsilon (2r-1,2s+1) \\
\epsilon (2r+1,2s+1) & \mapsto & \tau_1(bs+1) \epsilon (2r,2s+1).
\end{eqnarray*}
Similarly, along column $2r$ we use the resolution ${\bf{R}}_2(ar)$,
and along column $2r+1$ we use the resolution ${\bf{R}}_2(ar+1)$,
introducing a sign in the odd columns. The column maps are therefore
given by
\begin{eqnarray*}
\epsilon (2r,2s) & \mapsto & \gamma_2(ar) \epsilon (2r,2s-1) \\
\epsilon (2r,2s+1) & \mapsto & \tau_2(ar) \epsilon (2r,2s) \\
\epsilon (2r+1,2s) & \mapsto & - \gamma_2(ar+1) \epsilon (2r+1,2s-1) \\
\epsilon (2r+1,2s+1) & \mapsto & - \tau_2(ar+1) \epsilon (2r+1,2s).
\end{eqnarray*}
It is straightforward to verify that these maps indeed define a
double complex; all the four different types of squares commute. The
transpose of the matrices defining the maps in the resulting double
complex are given by
$$\left ( \begin{smallmatrix}
\gamma_1(0) & -\tau_2(as+1) & 0 & 0 & 0 & 0 & \cdots & 0 \cr \cr \cr
0& \tau_1(1) & \gamma_2(as)& 0&0&0& \cdots & 0 \cr \cr \cr 0&0&
\gamma_1(b) & -\tau_2(a[s-1]+1) & 0 &0& \cdots & 0 \cr \cr \cr
0&0&0&\tau_1(b+1) & \gamma_2(a[s-1]) &0& \cdots & 0 \cr \cr \cr
\vdots & \vdots && \ddots & \ddots & \ddots &\ddots & \vdots \cr
 &&&&&& \cr
 &&&&&& \cr
0&0& \cdots && 0 & \gamma_1(bs) & -\tau_2(1)&0 \cr \cr \cr 0&0&
\cdots &&0&0& \tau_1(bs+1)& \gamma_2(0) \end{smallmatrix} \right )
$$
for the map at stage $2(s+1)$, and
$$\left ( \begin{smallmatrix}
\tau_1(0) & \tau_2(as) & 0 & 0 &0 &0 & \cdots & 0 \cr \cr \cr 0&
\gamma_1(1) & -\gamma_2(a[s-1]+1) & 0&0&0 & \cdots & 0 \cr \cr \cr
0&0& \tau_1(b) & \tau_2(a[s-1]) & 0 & 0 & \cdots & 0 \cr \cr \cr
0&0&0& \gamma_1(b+1) & -\gamma_2(a[s-2]+1) &0& \cdots & 0 \cr \cr
\cr \vdots & \vdots && \ddots & \ddots & \ddots &\ddots & \vdots \cr
 &&&&&& \cr
 &&&&&& \cr
0&0& \cdots &&0& \gamma_1(b[s-1]+1) & -\gamma_2(1)&0 \cr \cr \cr
0&0& \cdots &&0&0& \tau_1(bs)& \tau_2(0) \end{smallmatrix} \right )
$$
for the map at stage $2s+1$.

Now for each $n \ge 0$ denote the generator $\epsilon (i,n-i)$ by
$f^n_i$, so that the $n$th bimodule in the total complex
$\mathbb{P}$ is
$$P_n = \bigoplus_{i=0}^n A^{\e} f^n_i,$$
the free $A^{\e}$-module of rank $n+1$ having generators $\{ f^n_0,
f^n_1, \dots, f^n_n \}$. Then the maps $P_n \xrightarrow{d_n}
P_{n-1}$ in $\mathbb{P}$ are given by
\begin{eqnarray*}
d_{2t} \colon f^{2t}_i & \mapsto & \left \{
\begin{array}{lcll}
 \gamma_2 ( \frac{ai}{2} ) f^{2t-1}_i & + & \gamma_1 (
 \frac{2bt-bi}{2} ) f^{2t-1}_{i-1} & \text{ for $i$ even} \\
 \\
 - \tau_2 ( \frac{ai-a+2}{2} ) f^{2t-1}_i & + & \tau_1 (
 \frac{2bt-bi-b+2}{2} ) f^{2t-1}_{i-1} & \text{ for $i$ odd}
\end{array} \right. \\
& \vspace{.5cm} & \\
d_{2t+1} \colon f^{2t+1}_i & \mapsto & \left \{
\begin{array}{llll}
\tau_2 ( \frac{ai}{2} ) f^{2t}_i & + & \gamma_1 ( \frac{2bt-bi+2}{2}
)
f^{2t}_{i-1} & \text{ for $i$ even} \\
\\
- \gamma_2 ( \frac{ai-a+2}{2} ) f^{2t}_i & + & \tau_1 (
\frac{2bt-bi+b}{2} ) f^{2t}_{i-1} & \text{ for $i$ odd,}
\end{array} \right.
\end{eqnarray*}
where we use the convention $f^n_{-1} = f^n_{n+1} = 0$. The
following result shows that the complex is exact when $q$ is not a
root of unity.

\begin{proposition}\label{exact}
The complex $\mathbb{P}$ is exact and is therefore a minimal
projective resolution
$$\mathbb{P} \colon \cdots \to P_2 \xrightarrow{d_2} P_1 \xrightarrow{d_1} P_0
\xrightarrow{\mu} A \to 0$$ of the left $A^{\e}$-module $A$.
\end{proposition}

\begin{proof}
We will show that the complex $\mathbb{P} \otimes_A k$ is exact and
a minimal projective resolution of the $A$-module $k$. Then the
arguments in \cite{Green} shows that the complex $\mathbb{P}$ is
exact.

When applying $- \otimes_A k$ to $A^{\e} = A \otimes_k A^{\op}$, the
elements $x$ and $y$ in $A^{\op}$ become zero, and so the elements
$\tau_i (s) \otimes 1$ and $\gamma_i (s) \otimes 1$ are just given
by
\begin{eqnarray*}
\tau_1 (s) \otimes 1 & = & - (x \otimes 1) \\
\tau_2 (s) \otimes 1 & = & -q^s(y \otimes 1) \\
\gamma_1 (s) \otimes 1 & = & (x^{a-1} \otimes 1) \\
\gamma_2 (s) \otimes 1 & = & q^{(b-1)s} ( y^{b-1} \otimes 1 ).
\end{eqnarray*}
We shall identify these elements with $-x, -q^sy, x^{a-1}$ and
$q^{(b-1)s}y^{b-1}$, respectively. Moreover, whenever the commutator
element $q$ is involved, its precise power does not affect the
dimensions of the vector spaces we are considering, so we shall
write $q^*$ for simplicity.

Fix a number $n \ge 0$. The free bimodule $P_n$ has generators
$\epsilon (i,j)$, with $n=i+j$ and $i,j \ge 0$. When the degree is
not ambiguous, we shall denote the element $\epsilon (i,j) \otimes 1
\in P_n \otimes_A k$ by $e_j$, and we shall denote the map $P_n
\otimes_A k \xrightarrow{d_n \otimes 1} P_{n-1} \otimes_A k$ by
$\widehat{d}_n$. Moreover, we denote by $U_i$ the left $A$-submodule
of $P_{n-1} \otimes_A k$ generated by $\widehat{d}_n (e_j)$, so that
$$\Im \widehat{d}_n = U_0 + \cdots + U_n \subseteq P_{n-1} \otimes_A
k.$$

We now compute the dimensions of these modules $U_i$. Assume first
that $n$ is even. Then
\begin{eqnarray*}
U_0 & = & Ax^{a-1}e_0 \\
U_i & = & A \left [ (q^*y)e_{i-1} + (q^*x)e_i \right ] \text{ for
odd } 0 < i < n \\
U_i & = & A \left [ (q^*y^{b-1})e_{i-1} + (q^*x^{a-1})e_i \right ]
\text{ for even } 0 < i < n \\
U_n & = & Ay^{b-1}e_{n-1},
\end{eqnarray*}
and so we see that $\dim U_0 =b, \dim U_n= a$, and otherwise $\dim U_i = ab-1$ and
$\dim U_j = a+b+1$ for $i$ odd and $j$ even. When $n$ is odd, then
\begin{eqnarray*}
U_0 & = & Axe_0 \\
U_i & = & A \left [ (-q^*y)e_{i-1} + (q^*x^{a-1})e_i \right ]
\text{ for odd } 0 < i < n \\
U_i & = & A \left [ (q^*y^{b-1})e_{i-1} + (q^*x)e_i \right ]
\text{ for even } 0 < i < n \\
U_n & = & Aye_{n-1},
\end{eqnarray*}
and so in this case we see that $\dim U_0 = b(a-1), \dim U_n =
a(b-1)$, and otherwise $\dim U_i = a(b-1)+1$ and $\dim U_j = b(a-1)+1$ for $i$ odd
and $j$ even.

Our aim is to compute the dimensions of various intersections and
sums obtained from the modules $U_i$. In order to do this, we need
the following fact: for any elements $z_1, z_2 \in A$, the
implication
\begin{equation*}\label{implikasjon}
z_1x^s = z_2y^t \Longrightarrow z_1 = v_1y^t+w_1x^{a-s} \text{ and
} z_2 = v_2x^s + w_2y^{b-t} \tag{$\dagger$}
\end{equation*}
holds, where $v_i,w_i$ are some elements in $A$ depending on $z_1$
and $z_2$. To see this, write $z_1 = g_0 + g_1y + \cdots +
g_{b-1}y^{b-1}$ and $z_2 = h_0 + h_1y + \cdots + h_{b-1}y^{b-1}$,
where the $g_i$ and $h_i$ are polynomials in $x$. Then
$$\sum_i h_iy^{t+i} = z_2y^t = z_1x^s = \sum_j
(q^{-js}g_jx^s)y^j,$$ and comparing the coefficients of $y^j$ we
find that $g_jx^s =0$ for $j<t$. Therefore, for these values of
$j$, the polynomial $g_j$ must be a multiple of $x^{a-s}$. Then we
can write $\sum_{j<t}g_jy^j = w_1x^{a-s}$ for some $w_1 \in A$,
giving
$$z_1 = \sum_{j<t}g_jy^j + \sum_{j \ge t}g_jy^j = w_1x^{a-s} +
v_1y^t,$$ where $v_1 = \sum_{j \ge t}g_jy^{j-t}$. This proves the
statement for $z_1$, the proof for $z_2$ is similar.

We now compute the intersections of pairs of the modules $U_i$.
Suppose $n$ is even, and fix an even integer $0 \le j \le n$. If
$u$ belongs to $U_j \cap U_{j+1}$, then there are elements $z_1,
z_2 \in A$ such that
$$u = z_1 \left [ (q^*y^{b-1})e_{j-1} + (q^*x^{a-1})e_j \right ] =
z_2 \left [ (q^*y)e_j + (q^*x)e_{j+1} \right ].$$ The coefficients
of $e_{j-1}$ and $e_{j+1}$ must be zero, whereas those of $e_j$
must be equal, giving $(z_1q^*)x^{a-1} = (z_2q^*)y$. By
(\ref{implikasjon}), there are elements $v_1, v_2, w_1, w_2 \in A$
such that
$$z_1q^* = v_1y + w_1x, \hspace{.5cm} z_2q^* = v_2x^{a-1} + w_2y^{b-1},$$
hence $u \in Ayx^{a-1}e_j$. Conversely, any element in
$Ayx^{a-1}e_j$ belongs to $U_j \cap U_{j+1}$, showing $U_j \cap
U_{j+1} = Ayx^{a-1}e_j$ and that the dimension of this
intersection is $b-1$. Similarly, we compute three other types of
intersections using the same method, and record everything in the
following table:
\begin{center}
\begin{tabular}{l|l|l|l}
n & j & \text{intersection} & \text{dimension} \\
\hline even & even & $U_j \cap U_{j+1} = Ayx^{a-1}e_j$ & $b-1$
\\
even & odd & $U_j \cap U_{j+1} = Ay^{b-1}xe_j$ & $a-1$
\\
odd & even & $U_j \cap U_{j+1} = Ayxe_j$ & $(a-1)(b-1)$
\\
odd & odd & $U_j \cap U_{j+1} = Ay^{b-1}x^{a-1}e_j$ & $1$
\end{tabular}
\end{center}

Next we show that the equality
$$(U_0 + U_1 + \cdots + U_s) \cap U_{s+1} = U_s \cap U_{s+1}$$
holds for any $s \ge 1$. Suppose first that both $n$ and $s$ are
even. The inclusion $U_s \cap U_{s+1} \subseteq (U_0 + U_1 +
\cdots + U_s) \cap U_{s+1}$ obviously holds, so suppose $u$ is an
element belonging to $(U_0 + U_1 + \cdots + U_s) \cap U_{s+1}$.
Then $u$ can be written as
\begin{eqnarray*}
u & = & z_0x^{a-1}e_0 + z_1 \left [ (q^*y)e_0 + (q^*x)e_1 \right ]
+ \cdots + z_s \left [ (q^*y^{b-1})e_{s-1} + (q^*x^{a-1})e_s
\right ] \\
& = & z_{s+1} \left [ (q^*y)e_s + (q^*x)e_{s+1} \right ],
\end{eqnarray*}
in which the coefficient of $e_{s+1}$ must be zero. Moreover, the
coefficients of $e_s$ must be equal, i.e.\ $(z_{s+1}q^*)y =
(z_sq^*)x^{a-1}$, and so from (\ref{implikasjon}) we see that
there exist elements $v,w \in A$ such that $z_{s+1} = vx^{a-1} +
wy^{b-1}$. This gives
$$u = (vx^{a-1} + wy^{b-1})q^*ye_s = vq^*x^{a-1}ye_s,$$
and we see directly that $u$ belongs to $U_s \cap U_{s+1}$. The
equality $(U_0 + U_1 + \cdots + U_s) \cap U_{s+1} = U_s \cap
U_{s+1}$ therefore holds when $n$ and $s$ are even, and the same
arguments show that the equality holds regardless of the parity of
$n$ and $s$.

Using what we just showed, an induction argument gives the
equality
$$\dim (U_0 + \cdots + U_s) = \sum_{i=0}^s \dim U_i -
\sum_{i=0}^{s-1} \dim ( U_i \cap U_{i+1} ).$$ Then by counting
dimensions, we see that the dimension of $\Im \widehat{d}_n$ is
given by
$$\dim \Im \widehat{d}_n = \left \{
\begin{array}{ll}
tab+1 & \text{when } n=2t \\
(t+1)ab-1 & \text{when } n=2t+1.
\end{array} \right.$$
The exactness of the complex $\mathbb{P} \otimes_A k$ now follows
easily; the image of $\widehat{d}_{n+1}$ is contained in the
kernel of $\widehat{d}_n$, and the dimension of $P_n \otimes_A k$
is $ab(n+1)$. It follows that $\Im \widehat{d}_{n+1}$ and $\Ker
\widehat{d}_n$ are of the same dimension.

As for minimality, it suffices to show that $\Im \widehat{d}_n$
does not have a projective summand. This follows from the
description of this module as the sum of the $U_i$. Namely, we see
directly that the element $y^{b-1}x^{a-1} \in A$ annihilates each
$U_i$ and therefore also $\Im \widehat{d}_n$.
\end{proof}

\section{Hochschild (co)homology}\label{computation}

Having obtained the bimodule resolution of $A = k \langle X,Y
\rangle / ( X^a, XY-qYX, Y^b )$, we turn now to its Hochschild
homology and cohomology groups. Let $B$ be a bimodule, and recall
that the Hochschild homology of $A$ with coefficients in $B$,
denoted $\HH_* (A,B)$, is the $k$-vector space
$$\HH_* (A,B) = \Tor^{A^{\e}}_*(B,A).$$
Dually, the Hochschild cohomology of $A$ with coefficients in $B$,
denoted
$\HH^* (A,B)$, is the $k$-vector space
$$\HH^* (A,B) = \Ext_{A^{\e}}^*(A,B).$$
Of particular interest is the case $B =A$, namely the Hochschild
homology and cohomology of $A$, denoted $\HH_*(A)$ and $\HH^*(A)$,
respectively. Now by viewing $A$ and $B$ as left $A^{\e}$-modules,
it follows from \cite[VI.5.3]{Cartan} that $D \left ( \HH^* (A,B)
\right )$ is isomorphic as a vector space to $\Tor^{A^{\e}}_*
\left ( D(B), A \right )$, where $D$ denotes the usual $k$-dual
$\Hom_k(-,k)$. In particular, by taking $B=A$ we see that $\dim_k
\HH^n(A) = \dim_k \Tor^{A^{\e}}_n \left ( D(A), A \right )$ for
all $n \geq 0$.

Our algebra $A$ is Frobenius; it is easy to check that the map $A
\xrightarrow{\phi} D(A)$ of left $A$-modules defined by
$$\phi (1) \colon \sum_{\substack{0 \leq j \leq b-1 \\ 0 \leq i \leq
a-1}} c_{j,i}y^jx^i \mapsto c_{b-1,a-1}$$ is an isomorphism. To
such a Frobenius isomorphism one can always associate a
$k$-algebra automorphism $A \xrightarrow{\nu} A$, a \emph{Nakayama
automorphism}, with the (defining) property that $w \cdot \phi(1)
= \phi(1) \cdot \nu (w)$ for all elements $w \in A$. In our case
the elements $x$ and $y$ generate $A$, and since $x \cdot \phi(1)
= \phi(1) \cdot q^{1-b}x$ and $y \cdot \phi(1) = \phi(1) \cdot
q^{a-1}y$, we see that the automorphism defined by
$$\nu \colon \begin{array}{l}
               x \mapsto q^{1-b}x \\
               y \mapsto q^{a-1}y
             \end{array} $$
is a Nakayama automorphism. The composite map $\phi \circ
\nu^{-1}$ is then a bimodule isomorphism between the right
$A^{\e}$-modules ${_{\nu}A_1}$ and $D(A)$, where the scalar action
on ${_{\nu}A_1}$ is given by $u \cdot (w_1 \otimes w_2) = \nu
(w_2) u w_1$. Consequently, we see that
$$\dim_k \HH^n(A) = \dim_k \Tor^{A^{\e}}_n ( {_{\nu}A_1}, A )$$
for all $n \geq 0$.

Now let $ \alpha, \beta \in k$ be nonzero scalars, and let $A
\xrightarrow{\psi} A$ be the automorphism defined by $x \mapsto
\alpha x$ and $y \mapsto \beta y$. Tensoring the deleted projective
bimodule resolution $\mathbb{P}_A$ with the right $A^{\e}$-module
${_{\psi}A_1}$, we obtain an isomorphism
$$\xymatrix{
\cdots \ar[r] & {_{\psi}A_1} \otimes_{A^{\e}} P_{n+1} \ar[d]^{\wr}
\ar[r]^{1 \otimes d_{n+1}} &
{_{\psi}A_1} \otimes_{A^{\e}} P_n \ar[d]^{\wr} \ar[r]^{1 \otimes d_n}
& {_{\psi}A_1}
\otimes_{A^{\e}} P_{n-1} \ar[d]^{\wr} \ar[r] & \cdots \\
\cdots \ar[r] & \oplus_{i=0}^{n+1} ({_{\psi}A_1}) e^{n+1}_i
\ar[r]^{\delta^{\psi}_{n+1}} &
\oplus_{i=0}^n ({_{\psi}A_1}) e^n_i \ar[r]^{\delta^{\psi}_n }
& \oplus_{i=0}^{n-1} ({_{\psi}A_1}) e^{n-1}_i \ar[r] & \cdots
}$$
of complexes, where $\{ e_0^n, e_1^n, \dots, e_n^n \}$ is the standard
generating set of $n+1$ copies of ${_{\psi}A_1}$. The map
$\delta^{\psi}_n$ is then given by
$$
\begin{array}{l}
\delta^{\psi}_{2t} \colon  y^ux^v e^{2t}_i \mapsto \vspace{2mm} \\
\left \{
\begin{array}{ll}
K_1^{\psi}(t,i,u,v)y^{u+b-1}x^ve^{2t-1}_i +
K_2^{\psi}(t,i,u,v)y^ux^{v+a-1}e^{2t-1}_{i-1} & \text{ for $i$ even}
\\
\\
\left [ q^{\frac{ai-a+2+2v}{2}} - \beta \right ] y^{u+1}x^ve^{2t-1}_i
+ \left [ \alpha q^{\frac{2bt-bi-b+2+2u}{2}} -1 \right ]
  y^ux^{v+1}e^{2t-1}_{i-1} & \text{
  for $i$ odd}
\end{array} \right. \\
\\
\delta^{\psi}_{2t+1} \colon y^ux^v e^{2t+1}_i \mapsto \vspace{2mm} \\
\left \{
\begin{array}{ll}
\left [ \beta - q^{\frac{ai+2v}{2}} \right ] y^{u+1} x^ v e^{2t}_i +
K_3^{\psi}(t,i,u,v) y^u x^{v+a-1} e^{2t}_{i-1} & \text{ for $i$ even}
\\
\\
K_4^{\psi}(t,i,u,v) y^{u+b-1} x^v e^{2t}_i + \left [ \alpha
  q^{\frac{2bt-bi+b+2u}{2}} - 1 \right ] y^u x^{v+1} e^{2t}_{i-1} &
\text{ for $i$ odd,}
\end{array} \right.
\end{array}
$$
where we use the convention $e^n_{-1} = e^n_{n+1} =0$. Here the
elements $K_j^{\psi}(t,i,u,v)$, which are scalars whose values
depend on the parameters $\psi,t,i,u,v$, are defined as follows:
\begin{eqnarray*}
K_1^{\psi}(t,i,u,v) & = & \sum_{j=0}^{b-1} q^{j (
    \frac{ai+2v}{2} )} \beta^{b-1-j}  \\
K_2^{\psi}(t,i,u,v) & = & \sum_{j=0}^{a-1} q^{j(
  \frac{2bt-bi+2u}{2} )} \alpha^j \\
K_3^{\psi}(t,i,u,v) & = & \sum_{j=0}^{a-1} q^{j(
  \frac{2bt-bi+2+2u}{2} )} \alpha^j \\
K_4^{\psi}(t,i,u,v) & = & \sum_{j=0}^{b-1} q^{j (
    \frac{ai-a+2+2v}{2} )} \beta^{b-1-j}
\end{eqnarray*}
Note that when $q$ is not a root of unity and the characteristic of
$k$ does not divide $a$ or $b$, then these scalars are all nonzero
when the automorphism $\psi$ is either the identity or the Nakayama
automorphism. For in this case the elements are of the form $q^s(1 +
q^m + q^{2m} + \cdots + q^{rm})$ for some $m,s \in \mathbb{Z}$ and
$r=a-1$ or $r=b-1$. When $m=0$ this element is nonzero since the
characteristic of $k$ does not divide $a$ or $b$, and if it was zero
for some $m \neq 0$ then $q$ would be a root of unity because of the
equality $(1 + q^m + q^{2m} + \cdots + q^{rm})(1-q^m) = 1-q^{(r+1)m}$.

In the following result we use this complex to compute the Hochschild
homology of our algebra $A$.

\begin{theorem}\label{homology}
When $q$ is not a root of unity, the Hochschild homology of $A$ is given
by
$$\dim_k \HH_n(A) = \left \{
\begin{array}{ll}
a+b-1 & \text{when } n=0 \\
a+b & \text{when } n \ge 1 \text{ and } \cha k \text{ divides both }
a \text{ and } b \\
a+b-1 & \text{when } n \ge 1 \text{ and } \cha k \text{ divides one
of } a \text{ and } b \\
a+b-2 & \text{when } n \ge 1 \text{ and } \cha k \text{ does not
  divide } a \text{ or } b
\end{array} \right.$$
\end{theorem}

\begin{proof}
We need to compute the homology groups of the above complex in the
case when $\psi$ is the identity automorphism on $A$, i.e.\ when
$\alpha = 1 = \beta$. We do this by computing $\Ker \delta_{2t}^1$ for
$t \ge 1$ and $\Ker \delta_{2t+1}^1$ for $t \geq 0$, and we treat
these two cases separately.

\subsection*{$\Ker \delta_{2t}^1$:}

$ $

The image under the map $\delta_{2t}^1$ of a basis vector $y^ux^v
e^{2t}_i \in \oplus_{i=0}^{2t}Ae^{2t}_i$ is given by
$$\begin{array}{ll}
K_1^1(t,i,u,v)y^{u+b-1}x^ve^{2t-1}_i +
K_2^1(t,i,u,v)y^ux^{v+a-1}e^{2t-1}_{i-1} & \text{ for $i$ even}
\\
\\
\left [ q^{\frac{ai-a+2+2v}{2}} - 1 \right ] y^{u+1}x^ve^{2t-1}_i
+ \left [ q^{\frac{2bt-bi-b+2+2u}{2}} -1 \right ]
  y^ux^{v+1}e^{2t-1}_{i-1} & \text{
  for $i$ odd.}
\end{array}$$
From the definition of the scalars $K_1^1$ and $K_2^1$ we see that
\begin{eqnarray*}
K_1^1(t,i,u,v) = 0 & \Leftrightarrow & i=0, \hspace{1mm} v=0,
\hspace{1mm} \cha k | b \\
K_2^1(t,i,u,v) = 0 & \Leftrightarrow & i=2t, \hspace{1mm} u=0,
\hspace{1mm} \cha k | a,
\end{eqnarray*}
and therefore we first compute the dimension of $\Ker \delta_{2t}^1$
under the assumption that the characteristic of $k$ does not divide
$a$ or $b$.

First we count the number of single basis vectors in
$\oplus_{i=0}^{2t}Ae_i^{2t}$ belonging to $\Ker \delta_{2t}^1$. For
even $i$, we have
\begin{eqnarray*}
\delta_{2t}^1 ( y^ux^v e^{2t}_i ) =0 \text{ for all even } i &
\Leftrightarrow & u+b-1 \geq b \text{ and } v+a-1 \geq a \\
& \Leftrightarrow & 1 \le u \le b-1 \text{ and } 1 \le v \le a-1,
\end{eqnarray*}
from which we obtain $(b-1)(a-1)(t+1)$ vectors (there are $t+1$ even
numbers in the set $\{ 0, 1, \dots, 2t \}$). For odd $i$, we have
\begin{eqnarray*}
\delta_{2t}^1 ( y^ux^v e^{2t}_i ) =0 \text{ for all odd } i &
\Leftrightarrow & u+1 \geq b \text{ and } v+1 \geq a \\
&\Leftrightarrow & u = b-1 \text{ and } v = a-1,
\end{eqnarray*}
giving $t$ vectors (there are $t$ odd numbers in the set $\{ 0, 1,
\dots, 2t \}$). Next we count the other single basis vectors which
are mapped to zero, starting with those for which $i$ is even. The
element $e^{2t-1}_{2t}$ is zero by definition, hence when $i=2t$ and
$v+a-1 \ge a$, that is when $1 \le v \le a-1$, we see that $y^ux^v
e^{2t}_i$ maps to zero. But the vectors for which $u$ is nonzero
were counted above, hence the new vectors are $x^ve^{2t}_{2t}$ for
$1 \le v \le a-1$. The element $e^{2t-1}_{-1}$ is also zero by
definition, hence when $i=0$ and $u+b-1 \ge b$, that is when $1 \le
u \le b-1$, we see that $y^ux^v e^{2t}_i$ also maps to zero. But
here the vectors for which $v$ is nonzero were counted above, and so
the new vectors are $y^ue^{2t}_0$ for $1 \le u \le b-1$. It is easy
to see that except for these $a+b-2$ new vectors, there is no other
single basis vector $y^ux^v e^{2t}_i$ in $\Ker \delta_{2t}^1$ for
which $i$ is even, since both $K_1^1(t,i,u,v)$ and $K_2^1(t,i,u,v)$
are always nonzero. Moreover, when $i$ is odd, neither $e^{2t-1}_i$
nor $e^{2t-1}_{i-1}$ are zero, and the coefficients
$[q^{\frac{ai-a+2+2v}{2}}-1]$ and $[q^{\frac{2bt-bi-b+2+2u}{2}}-1]$
are both nonzero. Hence in this case there are no new basis vectors
mapped to zero.

Now we count the number of nontrivial linear combinations of two or
more basis vectors in $\oplus_{i=0}^{2t}Ae_i^{2t}$ belonging to
$\Ker \delta_{2t}^1$. Let $i$ be even. If the first term of
$\delta_{2t}^1 ( y^ux^ve^{2t}_i )$ is nonzero, then the only way to
``kill'' it is to involve the second term of $\delta_{2t}^1 (
y^{u+b-1}x^{v-1}e^{2t}_{i+1})$. Thus to get a nontrivial linear
combination we see that $u,v$ and $i$ must satisfy $u=0, 1 \le v \le
a-1$ and $i=0, 2, \dots, 2t-2$. For these parameter values the
second term of $\delta_{2t}^1 ( y^ux^ve^{2t}_i )$ vanishes, as does
the first term of $\delta_{2t}^1 ( y^{u+b-1}x^{v-1}e^{2t}_{i+1})$.
Therefore, for suitable nonzero scalars $C(a,b,i,u,v)$, the linear
combination
$$x^ve^{2t}_i + C(a,b,i,u,v)y^{b-1}x^{v-1}e^{2t}_{i+1}$$
is mapped to zero for $1 \le v \le a-1$ and $i=0,2, \dots,
2t-2$, and there are $(a-1)t$ such elements. If the second term of
$\delta_{2t}^1 ( y^ux^ve^{2t}_i )$ is nonzero, then the only way to
``kill'' it is to involve the first term of $\delta_{2t}^1 (
y^{u-1}x^{v+a-1}e^{2t}_{i-1})$. To get a nontrivial linear
combination, the parameters $u,v,$ and $i$ must satisfy $1 \le u \le
b-1, v=0$ and $i=2,4, \dots, 2t$, and for these values the first term
of $\delta_{2t}^1 ( y^ux^ve^{2t}_i )$ and the second term of $\delta_{2t}^1 (
y^{u-1}x^{v+a-1}e^{2t}_{i-1})$ vanish. Thus, for suitable nonzero
scalars $C'(a,b,i,u,v)$, the linear combination
$$y^ue^{2t}_i + C'(a,b,i,u,v)y^{u-1}x^{a-1}e^{2t}_{i-1}$$
is mapped to zero for $1 \le u \le b-1$ and $i=2,4, \dots,
2t$, and there are $(b-1)t$ such elements.

We have now accounted for all the elements of $\Ker \delta_{2t}^1$
when the characteristic of $k$ does not divide $a$ or $b$. If the
characteristic of $k$ divides $a$, then we must add to our list the
element $e^{2t}_{2t}$. Similarly, if the characteristic of $k$
divides $b$, then we must add to our list the element $e^{2t}_0$.
Finally, if the characteristic of $k$ divides both $a$ and $b$, then
we must add both these two elements to our list (and they are
different elements since $t \ge 1$). Summing up, we see that the
total dimension of $\Ker \delta_{2t}^1$ is given by
$$\dim_k \Ker \delta_{2t}^1 = \left \{
\begin{array}{ll}
abt + ab -1 & \text{when } \cha k \text{ does not divide } a \text{ or
  }b \\
abt + ab+1 & \text{when } \cha k \text{ divides both } a \text{ and
} b \\
abt + ab & \text{otherwise.}
\end{array} \right.$$

\subsection*{$\Ker \delta_{2t+1}^1$:}

$ $

The image under the map $\delta_{2t+1}^1$ of a basis vector $y^ux^v
e^{2t+1}_i \in \oplus_{i=0}^{2t+1}Ae^{2t+1}_i$ is given by
$$\begin{array}{ll}
\left [ 1 - q^{\frac{ai+2v}{2}} \right ] y^{u+1} x^ v e^{2t}_i +
K_3^1(t,i,u,v) y^u x^{v+a-1} e^{2t}_{i-1} & \text{ for $i$ even}
\\
\\
K_4^1(t,i,u,v) y^{u+b-1} x^v e^{2t}_i + \left [
  q^{\frac{2bt-bi+b+2u}{2}} - 1 \right ] y^u x^{v+1} e^{2t}_{i-1} &
\text{ for $i$ odd.}
\end{array}$$
From the definition of the elements $K_3^1$ and $K_4^1$ we see that
they are always nonzero, contrary to the case above where there were
parameters for which $K_1^1$ and $K_2^1$ vanished. Therefore the
characteristic of $k$ does not matter when we compute the dimension of
$\Ker \delta_{2t+1}^1$.

We follow the same procedure as we did for $\Ker \delta_{2t}^1$;
first we count the number of single basis vectors in
$\oplus_{i=0}^{2t+1}Ae_i^{2t+1}$ belonging to $\Ker
\delta_{2t+1}^1$. For even $i$, we have
\begin{eqnarray*}
\delta_{2t+1}^1 ( y^ux^v e^{2t+1}_i ) =0 \text{ for all even } i &
\Leftrightarrow & u+1 \geq b \text{ and } v+a-1 \geq a \\
& \Leftrightarrow & u = b-1 \text{ and } 1 \le v \le a-1,
\end{eqnarray*}
resulting in $(a-1)(t+1)$ vectors (there are $(t+1)$ even numbers in
the set $\{0, 1, \dots, 2t+1 \}$). When $i$ is odd, we have
\begin{eqnarray*}
\delta_{2t+1}^1 ( y^ux^v e^{2t+1}_i ) =0 \text{ for all odd } i &
\Leftrightarrow & u+b-1 \geq b \text{ and } v+1 \geq a \\
& \Leftrightarrow & 1 \le u \le b-1 \text{ and } v = a-1,
\end{eqnarray*}
giving $(b-1)(t+1)$ vectors (there are $(t+1)$ odd numbers in the
set $\{0, 1, \dots, 2t+1 \}$). Next we count the other single basis
vectors in $\oplus_{i=0}^{2t+1}Ae_i^{2t+1}$ belonging to $\Ker
\delta_{2t+1}^1$, starting with those for which $i$ is even. The
element $e^{2t}_{-1}$ is zero, hence for $i=0$ the second term in
$\delta_{2t+1}^1 ( y^ux^v e^{2t+1}_i )$ vanishes. If now $v=0$, then
the coefficient $[ 1 - q^{\frac{ai+2v}{2}} ]$ vanishes, and
therefore the vector $y^u e^{2t+1}_0$ maps to zero for $0 \le u \le
b-1$. There are $b$ such vectors, and none of them was counted
above. Moreover, it is not hard to see that there is no other vector
$y^ux^v e^{2t+1}_i$ in $\Ker \delta_{2t+1}^1$ for which $i$ is even.
As for the case when $i$ is odd, the element $e^{2t}_{2t+1}$ is zero
by definition, and the coefficient $[ q^{\frac{2bt-bi+b+2u}{2}}-1 ]$
vanishes for $i=2t+1$ and $u=0$. Therefore the vector
$x^ve^{2t+1}_{2t+1}$ maps to zero for $0 \le v \le a-1$. These $a$
vectors have not been counted before, and $\Ker \delta_{2t+1}^1$
does not contain more vectors $y^ux^v e^{2t+1}_i$ for which $i$ is
odd.

At last we count the number of nontrivial linear combinations of two
or more basis vectors in $\oplus_{i=0}^{2t+1}Ae_i^{2t+1}$ belonging
to $\Ker \delta_{2t+1}^1$. Let $i$ be even, and suppose the first
term of $\delta_{2t+1}^1 ( y^ux^v e^{2t+1}_i )$ is nonzero. The only
way to cancel this term is to involve the second term of
$\delta_{2t+1}^1 ( y^{u+1}x^{v-1} e^{2t+1}_{i+1} )$. Now the first
term in the latter vanishes, as does the second term of
$\delta_{2t+1}^1 ( y^ux^v e^{2t+1}_i )$ since $v$ must be nonzero.
Thus for suitable nonzero scalars $C''(a,b,i,u,v)$, the element
$$y^ux^v e^{2t+1}_i + C''(a,b,i,u,v) y^{u+1}x^{v-1} e^{2t+1}_{i+1}$$
belongs to $\Ker \delta_{2t+1}^1$ when the parameters satisfy $0 \le u
\le b-2, 1 \le v \le a-1$ and $i=0,2, \dots, 2t$. There are
$(a-1)(b-1)(t+1)$ such elements. Finally, suppose the second term of
$\delta_{2t+1}^1 ( y^ux^v e^{2t+1}_i )$ is nonzero. To cancel it, we
must involve the first term in $\delta_{2t+1}^1 ( y^{u-b+1}x^{v+a-1}
e^{2t+1}_{i-1} )$, and so we see that the only possibility for $u$ and
$v$ is $u=b-1$ and $v=0$. Therefore, for suitable nonzero scalars
$C'''(a,b,i,u,v)$, the element
$$y^{b-1} e^{2t+1}_i + C'''(a,b,i,u,v) x^{a-1} e^{2t+1}_{i-1}$$
for $i=2,4, \dots, 2t$. There are $t$ such linear combinations.

All the elements of $\Ker \delta_{2t+1}^1$ are now accounted for, and
so when summing up we obtain the dimension of this vector space:
$$\dim_k \Ker \delta_{2t+1}^1 = abt + ab + a + b -1.$$

Using the identities $\dim_k \Ker \delta_n^1 + \dim_k \Im \delta_n^1 =
\dim_k A^{n+1} = (n+1)ab$, we can now calculate the Hochschild
homology of $A$. The dimension formula gives $\dim_k \Im
\delta_{2t+1}^1 = abt + ab -a-b+1$, in particular $\dim_k \Im
\delta_1^1 = 2ab-a-b+1$, giving
$$\dim_k \HH_0(A) = \dim_k A - \dim_k \Im \delta_1^1 = a+b-1.$$
Applying the formula to the results we obtained when computing $\Ker
\delta_{2t}^1,$ we get
$$\dim_k \Im \delta_{2t+2}^1 = \left \{
\begin{array}{ll}
abt+ab+1 & \text{when } \cha k \text{ does not divide } a \text{ or } b \\
abt+ab-1 & \text{when } \cha k \text{ divides both } a \text{ and } b\\
abt+ab & \text{otherwise,}
\end{array} \right.$$
and so by calculating $\dim_k \HH_n(A) = \dim_k \Ker \delta_n^1 -
\dim_k \Im \delta_{n+1}^1$ for $n \ge 1$ we get
$$\dim_k \HH_n(A) = \left \{
\begin{array}{ll}
a+b-2 & \text{when } \cha k \text{ does not divide } a \text{ or } b \\
a+b & \text{when } \cha k \text{ divides both } a \text{ and } b \\
a+b-1 & \text{otherwise.}
\end{array} \right.$$
This completes the proof.
\end{proof}

In particular, since $a$ and $b$ are both at least $2$, the Hochschild
homology of $A$ does not vanish in high degrees (or in any
degree). In \cite{Han} this was conjectured by Han to hold for all
finite dimensional algebras of infinite global dimension, and in the
same paper it was proved that this conjecture holds for monomial
algebras.

The converse of this conjecture always holds; if the global dimension
of an algebra is finite, then the algebra has finite projective
dimension as a bimodule, and hence its Hochschild homology vanishes in
high degrees. The same holds of course for Hochschild cohomology, and
in \cite{Happel}, following this easy observation, Happel remarked
that ``the converse seems to be not known''. Thus the cohomology
version of Han's conjecture came to be known as ``Happel's
question''. However, this cohomology version is false in general; it
was proved in \cite{Buchweitz} that there do exist finite dimensional
algebras of infinite global dimension for which Hochschild cohomology
vanishes in high degrees. The counterexample used in the paper was
precisely our algebra $A$ with $a=2=b$, and the following result shows
that the same holds for arbitrary $a$ and $b$. Contrary to the
homology case, the dimensions of the cohomology groups do not depend
on the characteristic of $k$.

\begin{theorem}\label{cohomology}
When $q$ is not a root of unity, the Hochschild cohomology of $A$ is given by
$$\dim_k \HH^n (A) = \left \{
   \begin{array}{ll}
      2 & \text{for } n = 0 \\
      2 & \text{for } n = 1 \\
      1 & \text{for } n = 2 \\
      0 & \text{for } n \ge 3.
   \end{array} \right.$$
In particular, the Hochschild cohomology of $A$ vanishes in high degrees.
\end{theorem}

\begin{proof}
It is well known and easy to see that in general $\HH^0(A)$ is
isomorphic to the center of $A$, that is, the subalgebra $\{ w \in A
\mid wz = zw \text{ for all } z \in A \}$. The center of our algebra
$A$ is the vector space spanned by the ``first'' and the ``last''
elements in its basis, namely the elements $1$ and
$y^{b-1}x^{a-1}$. Hence $\HH^0(A)$ is $2$-dimensional.

To compute the Hochschild cohomology groups of positive degree, we
compute the homology of the complex obtained prior to Theorem
\ref{homology}, in the case when $\psi$ is the Nakayama automorphism
$\nu$. In this case the scalars $\alpha$ and $\beta$ are given by
$\alpha = q^{1-b}$ and $\beta = q^{a-1}$. We apply the same method as
we did when computing homology; we compute $\Ker \delta^{\nu}_{2t}$
for $t \ge 1$ and $\Ker \delta^{\nu}_{2t+1}$ for $t \ge 0$, treating
the two cases separately.

\subsection*{$\Ker \delta_{2t}^{\nu}$:}

$ $

The result when applying the map $\delta_{2t}^{\nu}$ to a basis vector
$y^ux^ve^{2t}_i \in \oplus_{i=0}^{2t} \left ( {_{\nu}A_1} \right
)e^{2t}_i$ is given by
$$\begin{array}{ll}
K_1^{\nu}(t,i,u,v)y^{u+b-1}x^ve^{2t-1}_i +
K_2^{\nu}(t,i,u,v)y^ux^{v+a-1}e^{2t-1}_{i-1} & \text{ for $i$ even}
\\
\\
\left [ q^{\frac{ai-a+2+2v}{2}} - q^{a-1} \right ] y^{u+1}x^ve^{2t-1}_i
+ \left [ q^{\frac{2bt-bi-3b+4+2u}{2}} -1 \right ]
  y^ux^{v+1}e^{2t-1}_{i-1} & \text{
  for $i$ odd.}
\end{array}$$
From the definition of the elements $K_1^{\nu}$ and $K_2^{\nu}$ we see
that
\begin{eqnarray*}
K_1^{\nu}(t,i,u,v) = 0 & \Leftrightarrow & i=0, \hspace{1mm} v=a-1,
\hspace{1mm} \cha k | b \\
K_2^{\nu}(t,i,u,v) = 0 & \Leftrightarrow & i=2t, \hspace{1mm} u=b-1,
\hspace{1mm} \cha k | a,
\end{eqnarray*}
and so we first compute the dimension of $\Ker \delta_{2t}^{\nu}$ in
the case when the characteristic of $k$ does not divide $a$ or $b$.

First we count the number of single basis vectors in
$\oplus_{i=0}^{2t}({_{\nu}A_1})e_i^{2t}$ belonging to $\Ker
\delta_{2t}^{\nu}$. As in the homology case, we have
\begin{eqnarray*}
\delta_{2t}^{\nu} ( y^ux^v e^{2t}_i ) =0 \text{ for all even } i &
\Leftrightarrow & u+b-1 \geq b \text{ and } v+a-1 \geq a \\
& \Leftrightarrow & 1 \le u \le b-1 \text{ and } 1 \le v \le a-1, \\
\delta_{2t}^{\nu} ( y^ux^v e^{2t}_i ) =0 \text{ for all odd } i &
\Leftrightarrow & u+1 \geq b \text{ and } v+1 \geq a \\
&\Leftrightarrow & u = b-1 \text{ and } v = a-1,
\end{eqnarray*}
from which we obtain $(b-1)(a-1)(t+1) +t$ vectors. Next we count the
other single basis vectors in
$\oplus_{i=0}^{2t}({_{\nu}A_1})e_i^{2t}$ belonging to $\Ker
\delta_{2t}^{\nu}$. Since $K_1^{\nu}$ and $K_2^{\nu}$ are always
nonzero, the number of such vectors for which $i$ is even is the
same as in the homology case, namely $a+b-2$. As for the vectors for
which $i$ is odd, it is no longer true that the coefficients are
always nonzero. The coefficient $[q^{\frac{ai-a+2+2v}{2}} -
q^{a-1}]$ vanishes when $i=1$ and $v=a-2$, whereas
$[q^{\frac{2bt-bi-3b+4+2u}{2}} -1]$ vanishes when $i=2t-1$ and
$u=b-2$. Both these cases will occur, since $t$ is at least $1$ when
we compute $\Ker \delta_{2t}^{\nu}$. However, these coefficients
need to vanish \emph{simultaneously} for the basis vector to belong
to $\Ker \delta_{2t}^{\nu}$, and this only happens when $t=1$, since
then $2t-1=1$. Thus, when $t=1$ the vector $y^{b-2}x^{a-2}e^2_1$
maps to zero, whereas when $t \ge 2$ there are no new basis vectors
in $\Ker \delta_{2t}^{\nu}$ for which $i$ is odd.

Now we count the number of nontrivial linear combinations of two or
more basis vectors in $\oplus_{i=0}^{2t}({_{\nu}A_1})e_i^{2t}$
belonging to $\Ker \delta_{2t}^{\nu}$. These elements are precisely
the same as in the homology case, and we do not encounter problems
because of the ``new'' basis vector in $\Ker \delta_2^{\nu}$ we
obtained above. Therefore the number of such linear combinations is
$(a-1)t + (b-1)t$.

We now look at what happens when the characteristic of $k$ divides
$a$ or $b$. If $\cha k$ divides $a$, then we must add the vector
$x^{a-1}e^{2t}_0$ to the list of single basis vectors mapped to
zero. However, this vector already appears in one of the nontrivial
linear combinations, hence it does not contribute to the total
dimension. Similarly, when $\cha k$ divides $b$, then the new vector
$y^{b-1}e^{2t}_{2t}$ belongs to the list of single basis vectors
mapped to zero. But again this vector already appears in one of the
nontrivial linear combinations, and it will therefore not contribute
to the total dimension. This argument is still valid if $\cha k$
divides both $a$ and $b$. This shows that the dimension of $\Ker
\delta_{2t}^{\nu}$ is independent of the characteristic of $k$.

In total we see that the dimension of $\Ker \delta_{2t}^{\nu}$ is
almost the same as it was in the homology case when the characteristic
of $k$ did not divide $a$ or $b$; we need one additional vector when
$t=1$. Therefore the dimension is given by
$$\dim_k \Ker \delta_{2t}^{\nu} = \left \{
\begin{array}{ll}
  2ab & \text{when } t=1 \\
  abt +ab-1 & \text{when } t \ge 2.
\end{array} \right.$$

\subsection*{$\Ker \delta_{2t+1}^{\nu}$:}

$ $

The image under the map $\delta_{2t+1}^{\nu}$ of a basis vector $y^ux^v
e^{2t+1}_i \in \oplus_{i=0}^{2t+1} \left ( {_{\nu}A_1} \right )
e^{2t+1}_i$ is given by
$$\begin{array}{ll}
\left [ q^{a-1} - q^{\frac{ai+2v}{2}} \right ] y^{u+1} x^ v e^{2t}_i +
K_3^{\nu}(t,i,u,v) y^u x^{v+a-1} e^{2t}_{i-1} & \text{ for $i$ even}
\\
\\
K_4^{\nu}(t,i,u,v) y^{u+b-1} x^v e^{2t}_i + \left [
  q^{\frac{2bt-bi-b+2+2u}{2}} - 1 \right ] y^u x^{v+1} e^{2t}_{i-1} &
\text{ for $i$ odd.}
\end{array}$$
Now from the definition of the scalars $K_3^{\nu}$ and $K_4^{\nu}$ we
see that $K_3^{\nu}$ is always nonzero, while we have
$$K_4^{\nu}(t,i,u,v)=0 \Leftrightarrow i=1, \hspace{1mm} v=a-2,
\hspace{1mm} \cha k | b.$$ Therefore we first compute the dimension
of $\Ker \delta_{2t+1}^{\nu}$ under the assumption that the
characteristic of $k$ does not divide $b$.

First we count the number of single basis vectors in
$\oplus_{i=0}^{2t+1}({_{\nu}A_1})e_i^{2t+1}$ belonging to $\Ker
\delta_{2t+1}^{\nu}$. As in the homology case, we have
\begin{eqnarray*}
\delta_{2t+1}^{\nu} ( y^ux^v e^{2t+1}_i ) =0 \text{ for all even } i
&
\Leftrightarrow & u+1 \geq b \text{ and } v+a-1 \geq a \\
& \Leftrightarrow & u = b-1 \text{ and } 1 \le v \le a-1, \\
\delta_{2t+1}^{\nu} ( y^ux^v e^{2t+1}_i ) =0 \text{ for all odd } i
&
\Leftrightarrow & u+b-1 \geq b \text{ and } v+1 \geq a \\
& \Leftrightarrow & 1 \le u \le b-1 \text{ and } v = a-1,
\end{eqnarray*}
from which we obtain $(a-1)(t+1) + (b-1)(t+1)$ vectors. Next we
count the other single basis vectors in
$\oplus_{i=0}^{2t+1}({_{\nu}A_1})e_i^{2t+1}$ belonging to $\Ker
\delta_{2t+1}^{\nu}$, treating first the ones for which $i$ is even.
When $i=0$ the second term of $\delta_{2t+1}^{\nu}
(y^ux^ve^{2t+1}_i)$ vanishes, and the first term then vanishes if
$u=b-1$ or $v=a-1$. Some of these vectors are among the ones counted
above, the new ones are $y^{b-1}e^{2t+1}_0$ and
$y^ux^{a-1}e^{2t+1}_0$ for $0 \le u \le b-2$. Except for these $b$
elements, there are no other single basis elements in $\Ker
\delta_{2t+1}^{\nu}$ for which $i$ is even. As for those for which
$i$ is odd, we see that the first term of $\delta_{2t+1}^{\nu}
(y^ux^ve^{2t+1}_i)$ vanishes when $i=2t+1$. In this case the second
term vanishes if $u=b-1$ or $v=a-1$, and of these vectors the ones
which have not been counted before are the $a$ elements
$x^{a-1}e^{2t+1}_{2t+1}$ and $y^{b-1}x^ve^{2t+1}_{2t+1}$ for $0 \le
v \le a-2$. It is not hard to see that $\Ker \delta_{2t+1}^{\nu}$
does not contain any other element $y^ux^ve^{2t+1}_i$ for which $i$
is odd.

Finally we count the number of nontrivial linear combinations of two
or more basis elements in
$\oplus_{i=0}^{2t+1}({_{\nu}A_1})e_i^{2t+1}$ belonging to $\Ker
\delta_{2t+1}^{\nu}$. In the homology case, these were
$$y^ux^v e^{2t+1}_i + C''(a,b,i,u,v) y^{u+1}x^{v-1} e^{2t+1}_{i+1}$$
for $0 \le u \le b-2, 1 \le v \le a-1$ and $i=0,2, \dots, 2t$, and
$$y^{b-1} e^{2t+1}_i + C'''(a,b,i,u,v) x^{a-1} e^{2t+1}_{i-1}$$
for $i=2,4, \dots, 2t$, where $C''$ and $C'''$ are suitable
scalars. The $t$ latter elements also belong to $\Ker
\delta_{2t+1}^{\nu}$, but among the $(a-1)(b-1)(t+1)$ first elements
there are some combinations that are not mapped to zero. Namely, we
must discard the $b-1$ elements for which $i=0$ and $v=a-1$, since we
showed above that $y^ux^{a-1}e^{2t+1}_0$ maps to zero for $0 \le u \le
b-2$. Similarly, we must discard the $a-1$ combinations for which
$i=2t$ and $u=b-2$, since $y^{b-1}x^ve^{2t+1}_{2t+1}$ maps to zero for
$0 \le v \le a-2$. However, when $t=0$ then the situations $i=0$ and
$i=2t$ are the same, and the element $y^{b-2}x^{a-1}e^1_0 +
C''(a,b,i,u,v)y^{b-1}x^{a-2}e^1_1$ has been discarded twice. Thus the
total number of nontrivial linear combinations is
$(a-1)(b-1)(t+1)+t-(a-1)-(b-1)$ when $t \ge 1$, and one more when
$t=0$.

What happens when $\cha k$ divides $b$? The element $y^ux^{a-2}e^{2t+1}_1$
is not mapped to zero for any $u$, and it does not ``interfere''
with one of the nontrivial linear combinations. Hence the
dimension of $\Ker \delta_{2t+1}^{\nu}$ is also independent of the
characteristic of $k$.

In total we see that the dimension of $\Ker \delta_{2t+1}^{\nu}$
differs from that in the homology case since we need to subtract
$(a-1)+(b-1)$ when $t \ge 1$ and $(a-1)+(b-1)-1$ when $t=0$. Thus
the dimension is given by
$$\dim_k \Ker \delta_{2t+1}^{\nu} = \left \{
\begin{array}{ll}
ab+2 & \text{when } t=0 \\
abt+ab+1 & \text{when } t \ge 1.
\end{array} \right.$$

\sloppy We can now calculate the positive degree cohomology groups. We have
$\dim_k \Ker \delta_1^{\nu} = ab+2$, and since $\dim_k \Ker
\delta_2^{\nu} = 2ab$ we must have $\dim_k \Im \delta_2^{\nu} = ab$,
giving
$$\dim_k \HH^1(A) = \dim_k \Ker \delta_1^{\nu} - \dim_k \Im
\delta_2^{\nu} = 2.$$
Furthermore, since $\dim_k \Ker \delta_3^{\nu} = 2ab+1$, we must have
$\dim_k \Im \delta_3^{\nu} = 2ab-1$, giving
$$\dim_k \HH^2(A) = \dim_k \Ker \delta_2^{\nu} - \dim_k \Im
\delta_3^{\nu} = 1.$$
Similarly, direct computations show that the cohomology groups
$\HH^n(A)$ vanish when $n \ge 3$, thereby completing the proof.
\end{proof}

When the commutator element $q$ is a root of unity, then it is not
hard to see that the dimensions of infinitely many of the kernels in
the complex we used to compute (co)homology will increase. Therefore
the Hochschild homology of $A$ is still nonzero in all degrees,
while it is no longer true that all the higher Hochschild cohomology
groups vanish. We record this fact in the final result, which also
gives the multiplicative structure of the Hochschild cohomology ring
when $q$ is a root of unity.

\begin{theorem}\label{multiplication}
The Hochschild cohomology ring $\HH^*(A)$ is finite dimensional if
and only if $q$ is not a root of unity. When this is the case, the
algebra is isomorphic to the (five dimensional graded) fibre product
$$k[U]/(U^2) \times_k k \langle V,W \rangle / (V^2,VW+WV,W^2),$$
where $U$ is in degree zero and $V$ and $W$ are in degree one.
\end{theorem}

\begin{proof}
Suppose $q$ is not a root of unity. Recall first the initial part
$$P_2 \xrightarrow{d_2} P_1 \xrightarrow{d_1} P_0 \xrightarrow{\mu} A
\to 0$$
of the projective bimodule resolution of $A$, where $\mu$ is the
multiplication map. The maps $d_1$ and $d_2$ are defined on generators
as follows:
$$\begin{array}{lrcl}
d_1 \colon & f^1_0 & \mapsto & \left [ (1 \otimes y)-(y \otimes 1)
\right ] f^0_0 \\
& f^1_1 & \mapsto & \left [ (1 \otimes x)-(x \otimes 1) \right ] f^0_0
\\
\\
d_2 \colon & f^2_0 & \mapsto & \left [ (1 \otimes y^{b-1}) + (y
  \otimes y^{b-2}) + \cdots + (y^{b-1} \otimes 1) \right ] f^1_0 \\
& f^2_1 & \mapsto & \left [ q(1 \otimes x) - (x \otimes 1) \right ]
f^1_0 + \left [ (1 \otimes y) - q(y \otimes 1) \right ] f^1_1 \\
& f^2_2 & \mapsto & \left [ (1 \otimes x^{a-1}) + (x \otimes x^{a-2})
  + \cdots + (x^{a-1} \otimes 1) \right ] f^1_1.
\end{array}$$
Define two bimodule maps $g \colon P_1 \to A$ and $h
\colon P_1 \to A$ by
$$\begin{array}{lrcl}
g \colon & f^1_0 & \mapsto & y \\
& f^1_1 & \mapsto & 0 \\
\\
h \colon & f^1_0 & \mapsto & 0 \\
& f^1_1 & \mapsto & x. \\
\end{array}$$
One checks directly that $g \circ d_2 = 0 = h \circ d_2$, and that
neither of the two maps are liftable through $d_1$. Consequently they
represent the two basis elements of $\HH^1(A) = \Ext_{A^{\e}}^1(A,A)$.

We may identify the degree zero part of $\HH^*(A)$ with the center
of $A$, the two dimensional vector space spanned by the elements $1$
and $y^{b-1}x^{a-1}$. The latter element annihilates both $g$ and
$h$, hence $\HH^*(A)$ is isomorphic to the $k$-fibre product of the
algebra generated by $y^{b-1}x^{a-1}$ with the algebra generated by
$g$ and $h$. Since the Hochschild cohomology ring of a finite
dimensional algebra is always graded commutative (cf.\
\cite[Corollary 1.2]{Snashall}), both $g$ and $h$ square to zero.
Therefore, as $\HH^2(A)$ is one dimensional, we are done if we can
show that the product $hg \in \HH^2(A)$ is nonzero.

Define a bimodule map $g_0 \colon P_1 \to P_0$ by
$$\begin{array}{lrcl}
g_0 \colon & f^1_0 & \mapsto & (y \otimes 1) f^0_0 \\
& f^1_1 & \mapsto & 0.
\end{array}$$
It is not hard to see that there exists an element $w \in A^{\e}$ such
that the map $g_1 \colon P_2 \to P_1$ defined by
$$\begin{array}{lrcl}
g_1 \colon & f^2_0 & \mapsto & w f^1_0 \\
& f^2_1 & \mapsto & q(y \otimes 1)f^1_1 \\
& f^2_2 & \mapsto & 0
\end{array}$$
gives a commutative diagram
$$\xymatrix{
P_2 \ar[d]^{g_1} \ar[r]^{d_2} & P_1 \ar[d]^{g_0} \ar[dr]^{g} \\
P_1 \ar[r]^{d_1} & P_0 \ar[r]^{\mu} & A }$$
The product $hg \in \HH^2(A)$ is then represented by the composite map
$h \circ g_1$, under which the images of the generators in $P_2$ are
given by
$$\begin{array}{lrcl}
h \circ g_1 \colon & f^2_0 & \mapsto & 0 \\
& f^2_1 & \mapsto & qyx \\
& f^2_2 & \mapsto & 0.
\end{array}$$
This map is not liftable through $d_2$, and therefore it represents a
nonzero element of $\HH^2(A)$. Consequently the product $hg$ is nonzero.
\end{proof}


\begin{thebibliography}{EHSST}
\bibitem[AGP]{Avramov}L.\ Avramov, V.\ Gasharov, I.\ Peeva,
  \emph{Complete intersection dimension}, Publ.\ Math.\ I.H.E.S.\ 86
  (1997), 67-114.
\bibitem[BEH]{Benson}D.\ Benson, K.\ Erdmann, M.\ Holloway, \emph{Rank
    varieties for a class of finite-dimensional local algebras}, J.\
  Pure Appl.\ Algebra 211 (2007), no.\ 2, 497-510.
\bibitem[BGMS]{Buchweitz}R.-O.\ Buchweitz, E.\ Green, D.\ Madsen,
\O.\ Solberg, \emph{Finite Hochschild cohomology without finite
global dimension}, Math.\ Res.\ Lett.\ 12 (2005), no.\ 5-6, 805-816.
\bibitem[CaE]{Cartan}H.\ Cartan, S.\ Eilenberg, \emph{Homological
    Algebra}, Princeton University Press, 1956.
\bibitem[EHSST]{Erdmann}K.\ Erdmann, M.\ Holloway, N.\ Snashall,
  {\O}.\ Solberg, R.\ Taillefer, \emph{Support varieties for
    selfinjective algebras}, K-theory 33 (2004), 67-87.
\bibitem[GrS]{Green}E.\ Green, N.\ Snashall, \emph{Projective
bimodule resolutions of an algebra and vanishing of the second
Hochschild cohomology group}, Forum Math.\ 16 (2004), no.\ 1, 17-36.
\bibitem[Han]{Han}Y.\ Han, \emph{Hochschild (co)homology dimension},
  J.\ London Math.\ Soc.\ (2) 73 (2006), no.\ 3, 657-668.
\bibitem[Hap]{Happel}D.\ Happel, \emph{Hochschild cohomology of
finite-dimensional algebras}, in \emph{S{\'e}minaire d'Alg{\`e}bre
Paul Dubreil et Marie-Paul Malliavin, 39{\`e}me Ann{\'e}e (Paris,
1987/1988)}, Lecture Notes in Mathematics 1404 (1989), 108-126.
\bibitem[Hol]{Holm}T.\ Holm, \emph{Hochschild cohomology rings of algebras
$k[X]/(f)$}, Beitr{\"a}ge Algebra Geom.\ 41 (2000), no.\ 1, 291-301.
\bibitem[Man]{Manin}I.\ Manin, \emph{Some remarks on Koszul algebras
    and quantum groups}, Ann.\ Inst.\ Fourier (Grenoble) 37 (1987),
  191-205.
\bibitem[SnS]{Snashall}N.\ Snashall, {\O}.\ Solberg, \emph{Support
    varieties and Hochschild cohomology rings}, Proc.\ London Math.\
  Soc.\ 88 (2004), 705-732.
\end{thebibliography}
\end{document}